\documentclass[12pt]{amsart}
\usepackage{geometry}                
\usepackage{graphicx}
\usepackage{amssymb}
\usepackage{epstopdf}
\usepackage{amsthm}
\usepackage{mathscinet}
\usepackage{algorithm}
\usepackage{algorithmic}
\usepackage{tikz}
\usepackage[foot]{amsaddr}

\newtheorem{theorem}{Theorem}
\newtheorem{lemma}[theorem]{Lemma}
\newtheorem{corollary}[theorem]{Corollary}

\def\calG{\mathcal G}
\def\calI{\mathcal I}
\def\wt{\mathrm{wt}}
\def\vol{\mathop{\mathrm{vol}}}

\def\Ex{\mathop{\mathbb{E\null}}}
\def\Pr{\mathop{\mathbb{P\null}}}

\tikzstyle{vertex} = [draw, shape=circle,minimum size=1mm,  inner sep=1pt, fill]

\title[Hard-core model on $H$-free graphs]{Glauber dynamics for the hard-core model\\on bounded-degree $H$-free graphs}
\author[Mark Jerrum]{Mark Jerrum}
\address{School of Mathematical Sciences, Queen Mary, University of London, Mile End Road, London E1~4NS.}
\email{m.jerrum@qmul.ac.uk}

\begin{document}

\begin{abstract}
The hard-core model has as its configurations the independent sets of some graph instance~$G$. The probability distribution on independent sets is controlled by a `fugacity' $\lambda>0$, with higher $\lambda$ leading to denser configurations.  We investigate the mixing time of Glauber (single-site) dynamics for the hard-core model on restricted classes of bounded-degree graphs in which a particular graph~$H$ is excluded as an induced subgraph. If $H$ is a subdivided claw then, for all $\lambda$, the mixing time is $O(n\log n)$, where $n$ is the order of~$G$.  This extends a result of Chen and Gu for claw-free graphs.  When $H$ is a path, the set of possible instances is finite.  For all other $H$, the mixing time is exponential in~$n$ for sufficiently large $\lambda$, depending on $H$ and the maximum degree of~$G$. 
\end{abstract}

\maketitle

\section{Introduction}
The complete bipartite graph $K_{1,3}$ is known as a claw.  A graph is said to be claw-free if it does not contain a claw as an induced subgraph. The hard-core model is specified by a graph~$G$ and a fugacity $\lambda>0$.  These inputs define a probability distribution on independent sets of a graph, the so-called hard-core distribution.  Higher values of $\lambda$ favour larger independent sets, leading to denser configurations.  (Precise definitions will be given presently.) 

Glauber dynamics defines an ergodic Markov chain on independent sets in a graph, whose stationary distribution is the hard-core distribution. We are interested in analysing the mixing time (time to near stationarity) of Glauber dynamics.  With increasing fugacity $\lambda$ the density of configurations increases and we expect the mixing time to increase too. For most graph classes, for example graphs of bounded degree, which have been extensively studied~\cite{ChenLiuVigodaContraction,Sly,Weitz}, there is a critical $\lambda$ above which the mixing time becomes exponential in the size of the graph.  However, in his PhD thesis, Matthews~\cite{Matthews} showed that, for any fixed $\lambda$, however large, the mixing time of Glauber dynamics on a claw-free graph~$G$ is polynomial in $n=|V(G)|$.  

Recently, Chen and Gu~\cite{ChenGu} showed (amongst other things) that, with the additional assumption that $G$ has bounded degree $\Delta$, the mixing time of Glauber dynamics on claw-free graphs is in fact $O(n\log n)$.  The constant implicit in the O-notation depends on $\lambda$ and $\Delta$. We call this mixing `optimal', since $O(n\log n)$ time is necessary for Glauber dynamics to visit every vertex in~$G$, by the coupon-collector argument.  The intuition that $O(n\log n)$ is optimal can be formalised, though doing so is surprisingly challenging~\cite{HayesSinclair}. Prior to Chen and Gu's work, optimal mixing for matchings in bounded-degree graphs --- a related result --- had been shown by Chen, Liu and Vigoda~\cite{ChenLiuVigodaOptimal}. 

It is natural to wonder whether excluding graphs other than the claw can lead to optimal mixing time for all $\lambda$.  There is an extensive literature on independent sets in graph families that exclude a fixed connected graph~$H$ as an induced subgraph.  (From now on we restrict attention to connected~$H$.)  In this line of work, subdivided claws play a crucial role~\cite{AbrishamiEtAl,ChudnovskyQuasi}.  A subdivided claw is obtained from a claw by subdividing the three edges by an arbitrary number of intermediate vertices.  Alternatively, a subdivided claw is a tree with a single vertex of degree~3, and all other vertices having degree 1 or~2.  For $1\leq i\leq j\leq k$, we write $S_{i,j,k}$ for the subdivided claw whose leaves are at distance $i$, $j$ and $k$ from the degree-3 vertex.  The smaller subdivided claws have special names:   $S_{1,1,1}$ (as we have seen) is the \textit{claw}, $S_{1,1,2}$ the \textit{fork}, $S_{1,2,2}$ is the \textit{E}, and $S_{1.2.3}$ is the \textit{skew star}.  See Figure~\ref{fig:stars}.
 
\begin{figure}
\begin{tikzpicture}[scale=0.7, semithick]
   \draw (-0.5,-0.87) node [vertex] (a) {};
   \draw (-0.5,0.87) node [vertex] (b) {};
   \draw (0,0) node [vertex] (c) {};
   \draw (1,0) node [vertex] (d) {};
   \draw (a) -- (c) -- (d);
   \draw (b) -- (c);
   
   \draw (2,-0.87) node [vertex] (fa) {};
   \draw (2,0.87) node [vertex] (fb) {};
   \draw (2.5,0) node [vertex] (fc) {};
   \draw (3.5,0) node [vertex] (fd) {};
   \draw (4.5,0) node [vertex] (fe) {};
   \draw (fa) -- (fc) -- (fd) -- (fe);
   \draw (fb) -- (fc);

   \draw (6.5,0) node [vertex] (ea) {};
   \draw (5.5,1) node [vertex] (eb) {};
   \draw (5.5,0) node [vertex] (ec) {};
   \draw (5.5,-1) node [vertex] (ed) {};
   \draw (6.5,-1) node [vertex] (ee) {};
   \draw (6.5,1) node [vertex] (ef) {};
   \draw (ea) -- (ec) -- (ed) -- (ee);
   \draw (ec) -- (eb) -- (ef);

   \draw (8,-0.87) node [vertex] (sa) {};
   \draw (8,0.87) node [vertex] (sb) {};
   \draw (8.5,0) node [vertex] (sc) {};
   \draw (9.5,0) node [vertex] (sd) {};
   \draw (10.5,0) node [vertex] (se) {};
   \draw (11.5,0) node [vertex] (sf) {};
   \draw (7.5,-1.74) node [vertex] (sg) {};
   \draw (sb) -- (sc) -- (sd) -- (se) -- (sf);
   \draw (sg) -- (sa) -- (sc);

\end{tikzpicture}
\caption{The claw, the fork, the E and the skew star.}
\label{fig:stars}
\end{figure}

It soon becomes apparent that the interesting cases for us are when the excluded graph $H$ is a subdivided claw.  Informally, the reason is this: if $H$ is not a subdivided claw or a path, then it contains either a vertex of degree greater than~3, or a cycle, or two vertices of degree~3 connected by a path.  We can exclude the first case by working with instance graphs~$G$ with maximum degree~3, and the other two cases by working with graphs whose edges have been `stretched' or subdivided by sufficiently many vertices.  This intuition is formalised in Section~\ref{sec:exponential}.  Note that excluding a path in a bounded-degree graph leads to a finite graph class. The (fairly routine) conclusion, expressed in Theorem~\ref{thm:torpid}, is that Glauber dynamics mixes in exponential time for sufficiently large $\lambda$ (depending on $\Delta$ and~$H$ but not on $n=|V(G)|$) when $H$ is not a subdivided claw or a path.    

Our main result, Theorem~\ref{thm:main} in Section~\ref{sec:exponential}, asserts that bounded-degree graphs that exclude any subdivided claw support optimal mixing.  Together with the matching negative result just mentioned, we obtain the following dichotomy.  If $\mathcal H$ is any collection of graphs, we say that a graph $G$ is \textit{$\mathcal H$-free} if $G$ is $H$-free (does not contain $H$ as an induced subgraph) for all $H\in\mathcal H$. 
\begin{theorem}\label{thm:composite}
Let $\mathcal H$ be any finite set of connected graphs, and $\Delta\geq3$. Consider the hard-core model on $\mathcal H$-free graphs of maximum degree~$\Delta$.  If $\mathcal H$ contains a subdivided claw or a path then Glauber dynamics mixes in optimal $O(n\log n)$ time, at any fixed fugacity $\lambda>0$, where $n$ is the number of vertices of~$G$. Otherwise, for some sufficiently large fixed $\lambda$ (depending on $\mathcal H$), the mixing time is exponential in~$n$ in the worst case, even when $\Delta=3$.
\end{theorem}
The above statement is essentially a composite of Theorems \ref{thm:main} and \ref{thm:torpid}, and  the (very short) proof can be found in Section~\ref{sec:exponential}. 

There is a formal similarity between this dichotomy and that of Abrishami, Chudnovsky, Dibek, Cemil and Rz\polhk{a}\.{z}ewski \cite{AbrishamiEtAl}, who studied the problem of finding the largest independent set in a bounded-degree $H$-free graph.  The current result is technically not as difficult.  This is to be expected:  the algorithm, for finding a largest independent set in a claw-free graph is quite involved, whereas that for sampling an independent set is quite simple:  simulate Glauber dynamics for $O(n\log n)$ steps.

It is interesting to note that claw-free graphs arise in nature as crystal structures, for example the kagome and pyrochlore lattices.   However, I am unaware of any real-life crystals that are free of subdivided claws without in fact being claw-free.

In related work, Chudnovsky and Seymour~\cite{ChudnovskySeymour} showed that the partition function of the hard-core model on claw-free graphs has only (negative) real roots.  The existence of a zero-free region that includes the positive real axis has implications for the existence of \emph{deterministic} algorithms for approximating the partition function, through work of Barvinok~\cite{Barvinok} and Patel and Regis~\cite{PatelRegts}.  For fork-free graphs, a zero-free region, smaller but still including the positive real axis, has been identified by Bencs~\cite{BencsPhD}.

Note that the restriction to bounded-degree graphs is crucial.  General fork-free graphs include the complete bipartite graphs of arbitrary sizes, which certainly do not support optimal or even polynomial-time mixing.  Additional restrictions other than bounded degree can lead to polynomial-time algorithms;  for example, Dyer, Jerrum  and M\"uller treat the case of fork-free, odd-hold free graphs (equivalent to fork-free perfect graphs)~\cite{DyerJerrumMuller}.

Finally, although we concentrate on the hard-core model here, a result of a similar flavour has been obtained for the antiferromagnetic Ising model.  For this model, Glauber dynamics has been shown to have optimal mixing time on the class of line graphs at any non-zero temperature however small (i.e., for any interaction strength however large).   This result is due to Chen, Liu and Vigoda~\cite{ChenLiuVigodaHolant}, strengthening the mixing time bound of Dyer, Heinrich, Jerrum and M\"uller~\cite{DyerHeinrichJerrumMuller} from polynomial to optimal. 

\section{Preliminaries}

Suppose $G$ is a graph and $\lambda>0$ a `fugacity'.  A set $I\subseteq V(G)$ is an \emph{independent set} in $G$ if no edge in $E(G)$ has both endpoints in~$I$. The \emph{weight} $\wt(I)$ of independent set $I$ is simply $\wt(I)=\lambda^{|I|}$.  Denote by $\calI_G$ the set of all independent sets in~$G$.  The \emph{hard-core distribution} $\mu_{G,\lambda}$ assigns probability $\mu_{G,\lambda}(I)=\wt(I)/Z(G,\lambda)$ to each $I\in\calI_G$, where $Z(G,\lambda)$ is the \emph{partition function}.   
$$
Z(G,\lambda)=\sum_{I\in\calI_G}\wt(I).
$$
We think of $\mu_{G,\lambda}$ being a distribution on $2^{V(G)}$ with support $\calI_G$.  

Glauber dynamics defines a Markov chain on the independent sets of $G$.  It is presented in Figure~\ref{fig:Glauber}, where $\Gamma_G(v)$ is used to denote the set of neighbours of~$v$ in~$G$.  This Markov chain is ergodic, and converges to the hard-core distribution on independent sets of~$G$.  

\begin{figure}
\begin{itemize}
\item Let the current state (independent set) be $I$. 
\item Select a vertex $v\in V(G)$ uniformly at random.  
\item If $\Gamma_G(v)\cap I=\emptyset$ then:
\begin{itemize} 
\item with probability $\lambda/(1+\lambda)$, set $I'\leftarrow I\cup\{v\}$;
\item with the remaining probability set $I'\leftarrow I\setminus\{v\}$;
\end{itemize}
\item else: set $I'\leftarrow I$.
\item The new state is $I'$.
\end{itemize}
\caption{Glauber dynamics for the hard-core model.}
\label{fig:Glauber}
\end{figure}

Denoting the $t$-step transition probabilities of this Markov chain by $P^t(\cdot,\cdot)$, convergence to stationarity may be measured by total variation distance (i.e, one half of the $\ell_1$ distance) at time $t$, i.e., 
$$
d(t)=\max_{I\in\calI_G}\big\{d_\mathrm{TV}(P^t(I,\,\cdot\,),\mu_{G,\lambda})\big\}.
$$
Then the \emph{mixing time} of the Markov chain is defined to be $t_\mathrm{mix}=\min\{t:d(t)\leq\frac14\}$.  The constant $\frac14$ is arbitrary, subject to lying in the interval $(0,\frac12)$.  It is standard~\cite[\S4.5]{LevinPeresWilmer} that convergence to stationary is exponential, when time is measured in units of $t_\mathrm{mix}$, specifically,  $d(\ell\,t_\mathrm{mix})\leq 2^{-\ell}$, for all $\ell\in\mathbb{N}$.  The central question is then: is $t_\mathrm{mix}$ bounded by a polynomial in $n=|V(G)|$ and, more particularly, is $t_\mathrm{mix}=O(n\log n)$ (with the constant of proportionality being a function of $\lambda$ and~$\Delta$)?

For a short while we change perspective very slightly and view independent sets in~$G$ as spin configurations in $\{0,1\}^{V(G)}$.  Naturally, the independent set $I\in\calI_G$ corresponds to the spin configuration $\sigma$, where 
$$
\sigma(v)=\begin{cases}0,&\text{if $v\notin I$};\\1,&\text{if $v\in I$}.\end{cases}
$$
Let $\beta_{G,\lambda}$ be the product distribution of independent Bernoulli random variables, with success probability $p=\lambda/(1+\lambda)$ living on the vertices of~$G$. The following simple fact will be important in what follows.

\begin{lemma}\label{lem:sd}
The distribution $\beta_{G,\lambda}$ stochastically dominates the distribution~$\mu_{G,\lambda}$.  That is, there is a coupling of random variables, $\sigma$ and $\sigma'$, distributed as $\mu_{G,\lambda}$ and $\beta_{G,\lambda}$, respectively, such that $\sigma\leq\sigma'$ in the product order on $\{0,1\}^{V(G)}$.
\end{lemma} 

\begin{proof}
A simple way to see this is by monotone coupling of Markov chains~\cite[\S3.10.3]{FriedliVelenik}.  Modify the Markov chain from Figure~\ref{fig:Glauber} by removing the test $\Gamma_G(v)\cap I=\emptyset$ and the `else' part of the conditional statement.  Thus the probabilistic update is done whatever the state of the vertices adjacent to~$v$.  This modified process clearly converges to the product distribution $\beta_{G,\lambda}$.  Now run the modified process starting from the all-1 configuration in parallel with Glauber dynamics starting from the all-0 configuration.  Use the same random choice of vertex $v$ at each step, and optimally couple the updates.  It is easy to see that the first process remains above the second in the product order.  The upper process converges to $\beta_{G,\lambda}$ and the lower to $\mu_{G,\lambda}$.
\end{proof}

\begin{corollary}\label{cor:marginalboundedness}
Let $G$ be a graph of maximum degree~$\Delta$, and $v$ be a vertex of $G$.  If $I$ is sampled from $\mu_{G,\lambda}$ then
$$
\frac\lambda{(1+\lambda)^{\Delta+1}}\leq \Pr(v\in I)\leq \frac\lambda{1+\lambda}.
$$
\end{corollary}

\begin{proof}
The upper bound is immediate from Lemma~\ref{lem:sd}.  For the lower bound, note that $\Pr(I\cap\Gamma(v)=\emptyset)\geq 1/(1+\lambda)^\Delta$, again from Lemma~\ref{lem:sd}, and that $\Pr(v\in I\mid I\cap\Gamma(v)=\emptyset)\geq \lambda/(1+\lambda)$.\end{proof}

Suppose $\mu$ and $\mu'$ are probability distributions on $\{0,1\}^{V(G)}$.  The Wasserstein (or Kantorovich-Rubinstein) distance between $\mu$ and $\mu'$ (specialised to this application) is defined to be
$$
W_1(\mu,\mu')=\inf_{\sigma\sim\mu,\sigma'\sim\mu'}d_\mathrm{H}(\sigma,\sigma'),
$$
where $d_\mathrm{H}$ denotes Hamming distance, and the infimum is over all couplings of random variables $\sigma,\sigma'$ with the appropriate distributions.  

For a graph $G$ and vertex $v\in V(G)$ let $\mu_{G,\lambda}^{(v,0)}$ (respectively, $\mu_{G,\lambda}^{(v,1)}$) denote the hard-core distribution on~$G$ at fugacity~$\lambda$ conditioned on vertex~$v$ being excluded from (respectively, included in) the independent set.  We will be using the powerful spectral independence machinery of Anari, Liu, Oveis Gharan and Vinzant~\cite{ALOV} and Anari, Liu and Oveis Gharan~\cite{ALO}, as sharpened by Cryan, Guo and Mousa~\cite{CryanGuoMousa} and Chen, Liu and Vigoda~\cite{ChenLiuVigodaOptimal}.  However, we can ease our task by using Chen and Gu's packaging of the machinery for situations such as ours.  Indeed, as we shall see presently, to show optimal mixing we just need to construct a coupling that achieves bounded Wasserstein distance $W_1\big(\mu_{G,\lambda}^{(v,0)},\mu_{G,\lambda}^{(v,1)}\big)$. 

A graph class is said to be \emph{hereditary} if it is closed under taking induced subgraphs.  Note that graph classes defined by excluding induced subgraphs, with or without a degree bound, are hereditary.

\begin{corollary}\label{cor:mixing}
Let $\calG$ be a hereditary class of graphs with degree bound $\Delta$, and $\lambda$ be positive.   Suppose there exists $\eta>0$ such that for all $G\in\calG$ and all $v\in V(G)$ we have $W_1(\mu_{G,\lambda}^{(v,0)},\mu_{G,\lambda}^{(v,1)})\leq\eta$, where $\mu_{G,\lambda}$ is the hard-core distribution with fugacity $\lambda$ on $V(G)$. Then the mixing time of Glauber dynamics for the hard-core model on $G\in\calG$ at fugacity~$\lambda$ is bounded above by $O(n\log n)$, where $n=|V(G)|$.  The constant implicit in the $O$-notation depends on  $\lambda$, $\Delta$ and $\eta$.
\end{corollary}

\begin{proof}
This follows from~\cite[Thm 9 and Lemma 10]{ChenGu},\footnote{The quoted results are stated for edge-based configurations rather than vertex-based, but this is inessential.} which builds on earlier work, such as \cite{ChenLiuVigodaOptimal}. In order to make the correspondence we need to discuss briefly the notion of `pinning'.  The distributions $\mu_{G,\lambda}^{(v,0)}$ and $\mu_{G,\lambda}^{(v,1)}$ are examples of distributions obtained by pinning, in this instance by pinning vertex $v$ to 0 and~1, respectively.  More generally, we may consider pinning a set of vertices, say $\Lambda\subset V(G)$, to values given by $\tau:\Lambda\to\{0,1\}$, resulting in a marginal distribution $\mu_{G,\lambda}^{(\Lambda,\tau)}$ on $V(G)\setminus\Lambda$.  Lemma 10 of~\cite{ChenGu} quantifies over pairs of pinnings $(\Lambda,\tau)$ and $(\Lambda,\tau')$ where $\tau$ and $\tau'$ differ at exactly one point.  However, the hard-core distribution has a pleasant property:  pinning a vertex~$v$ to~0 is equivalent to removing $v$ from the graph, while pinning~$v$ to~1 is equivalent to removing $v$ and all its neighbours.  So, in the case of the hard-core model we need only quantify over single vertex pinnings, since we are restricting attention to graph classes which are closed under taking induced subgraphs.

From our premisses, Lemma 10 of~\cite{ChenGu} tells us that that $\mu_{G,\lambda}$ is `$\eta$-spectrally bounded'.  For our purposes we don't need to define this concept as we are merely going to feed this information about $\mu_{G,\lambda}$ into Theorem~9 of~\cite{ChenGu}. The remaining premiss required for this theorem is that $\mu_{G,\lambda}$ is `$b$-marginally bounded'.  This just means that $\Pr(\sigma(v)=1)$ is bounded away from both 0 and 1 by a margin~$b$.  But this is exactly what is guaranteed by  Corollary~\ref{cor:marginalboundedness}, with $b=\min\{\lambda/(1+\lambda)^{\Delta+1}, 1/(1+\lambda)\}$.
\end{proof}

\section{$S_{1,1,t}$-free graphs (claw-free and upwards)}
Suppose $\Delta\geq3$ is a degree bound and $H$ a fixed graph.  Let $G$ be an $H$-free (i.e., excluding induced copies of~$H$) connected graph of maximum degree~$\Delta$.  Chen and Gu~\cite[Thm~4]{ChenGu} show optimal mixing of Glauber dynamics for the hard-core model on $G$, for any given fugacity~$\lambda>0$, when the excluded graph~$H$ is the claw.  In fact, optimal mixing also holds for $H=S_{1,1,t}$ with $t\geq2$, for a rather dull reason.  Define 
$$
\vol(\Delta,t)=1+\Delta\frac{(\Delta-1)^t-1}{\Delta-2},
$$
and note that $\vol(\Delta,t)$ bounds the number of vertices in a radius-$t$ ball around any vertex of a graph of maximum degree~$\Delta$.  

\begin{lemma}\label{lem:S111impliesS11t}
Suppose that $G$ is a connected graph of maximum degree $\Delta$, and that $t\geq2$.  If $|V(G)|>\vol(\Delta,t+1)$ and $G$ contains an induced claw, then $G$ contains an induced $S_{1,1,t}$.
\end{lemma}

\begin{proof}
Let $G$ be as in the statement and suppose $\{v,a,b,c\}\subseteq V(G)$ induces a claw centred at~$v$.  As $\vol(\Delta,t+1)$ is an upper bound on the number of vertices of $G$ within distance $t+1$ of $v$, there must exist a vertex $w$ at distance $t+2$ from~$v$. Consider a minimum length path from $\{v,a,b,c\}$ to $w$, which clearly has length at least $t+1$.  Let this path be $P=(w_0,w_1,\ldots,w_s)$, where $w_0\in\{v,a,b,c\}$, $w_s=w$ and $s\geq t+1$.  Consider the graph induced by the vertex set $U=\{v,a,b,c,w_1,w_2,\ldots, w_s\}$.  By minimality of $P$, the graph $G[U]$ contains edges $\{\{v,a\},\{v,b\},\{v,c\},\{w_1,w_2\}, \{w_2,w_3\},\ldots, \{w_{s-1},w_s\}\}$ together with some non-empty subset $A$ of $\{\{v,w_1\},\{a,w_1\},\{b,w_1\},\{c,w_1\}\}$.  The 15 possibilities for~$A$ fall into four cases, up to symmetry:
\begin{itemize}
\item $A=\{\{v,w_1\}\}$: the set $\{v,a,b,w_1,\ldots, w_{t}\}$ induces $S_{1,1,t}$;
\item $A=\{\{a,w_1\}\}$: the set $\{v,b,c,a,w_1,\ldots,w_{t-1}\}$ induces $S_{1,1,t}$;
\item $A=\{\{v,w_1\},\{a,w_1\}\}$: the set $\{v,b,c,w_1,\ldots,w_{t}\}$ induces $S_{1,1,t}$;
\item $A\supseteq\{\{a,w_1\},\{b,w_1\}\}$: the set $\{w_1,a,b,w_2,\ldots,w_{t+1}\}$ induces $S_{1,1,t}$.
\end{itemize}
The cases are exhaustive, completing the proof.
\end{proof}

\begin{corollary}\label{cor:finitevariation}
Restricted to connected graphs of maximum degree $\Delta$, the class of $S_{1,1,t}$-free graphs is a finite variation of the class of claw-free graphs, for any $t\geq2$.
\end{corollary}

\begin{corollary}
For any $\Delta\geq3$ any $t\geq1$ and any $\lambda>0$, Glauber dynamics for the hard-core model at fugacity~$\lambda$ has mixing time $O(n\log n)$ on the class of $S_{1,1,t}$-free graphs of maximum degree~$\Delta$.
\end{corollary}

\begin{proof}
This is immediate from \cite[Thm~4]{ChenGu} and Corollary~\ref{cor:finitevariation}.  
\end{proof}

\section{$S_{1,2,t}$-free graphs (E-free and upwards)}

Before heading off to larger graph classes it is instructive to go back to a precursor of~\cite{ChenGu}.  Van den Berg and Brouwer~\cite{vandenBergBrouwer} construct an explicit coupling for matchings in a bounded-degree graph that (in modern terms) bounds spectral independence.\footnote{Although this reference is closest to our application, the general concept of `disagreement percolation' appears earlier in work of van den Berg and Maes~\cite{vandenBergMaes} and van den Berg and Steif~\cite{vandenBergSteif}.}  This almost immediately gives optimal mixing of Glauber dynamics for the monomer-dimer model at all temperatures.  (Of course, van den Berg and Brouwer did not have the spectral independence technology, so they were only able to deduce low-degree polynomial mixing time.)  Their method extends easily to independent sets in claw-free graphs (though this was not stated in their paper).  In this section, we rederive the result of Chen and Gu, as the proof will provide a base from which we may progress to large subdivided claws.   

\begin{figure}
\begin{itemize}
\item Let $R$ be a random `red' independent set from $\mu_{G,\lambda}^{(v,1)}$.
\item Let $B$ be a random `blue' independent set from $\mu_{G,\lambda}^{(v,0)}$,  with $B$ independent of $R$.
\item Consider the subgraph $G[R\cup B]$ of $G$ induced by the union of the two independent sets.  Let $C\subseteq V(G)$ be the connected component of  $G[R\cup B]$ containing~$v$.  We refer to $C$ as a cluster.  
\item Modify the blue independent set by setting $B\leftarrow(B\cap C)\cup(R\setminus C)$.  Thus, $B$ is made to agree with~$B$ inside~$C$ and with~$R$ outside. The independent set $R$ remains unchanged.
\item $(R,B)$ is the required coupling.
\end{itemize}
\caption{The coupling of van den Berg and Brouwer.  Here, $G$ is a graph, $v\in V(G)$ and $\lambda>0$.}
\label{fig:couple}
\end{figure}
The coupling of van den Berg and Brouwer, translated to independent sets, is given in Figure~\ref{fig:couple}.

\begin{lemma}\label{lem:couple}
Let $G$ be a connected claw-free graph, $v\in V(G)$ be a vertex, and $\lambda>0$. Then the coupling of Figure~\ref{fig:couple} is correct:  $R$ is a sample from  $\mu_{G,\lambda}^{(v,1)}$ and $B$ is a sample from $\mu_{G,\lambda}^{(v,0)}$.
\end{lemma}

\begin{proof}
Since the independent set $R$ was never modified, it is clearly a sample from $\mu_{G,\lambda}^{(v,1)}$.  

For $B$, we argue as follows.  Think of the pair $R,B$ being selected as follows.  First we select the cluster $C\subseteq V(G)$, as if it had been formed from a pair of independent sets $R,B$ as above.  Since $G[C]$ is a connected bipartite graph, and $v\in R$, we can reconstruct the portion of $R$ and $B$ lying within $C$.  Denote by 
$$
\partial C=\{w\in V(G)\setminus C: \{w,u\}\in E(G)\text{ for some }u\in C\}
$$ 
the boundary of~$C$.   A subset $R_\mathrm{ext}\subseteq V(G)\setminus C$ extends $R_\mathrm{int}=R\cap C$ to an independent set $R=R_\mathrm{int}\cup R_\mathrm{ext}$ on the whole of $V(G)$ if and only if $R_\mathrm{ext}$ is an independent set of the graph $G[V(G)\setminus(C\cup\partial C)]$.  (A vertex in $R_\mathrm{ext}$ certainly cannot be adjacent to a red vertex in the cluster $C$, and it cannot be adjacent to a blue vertex by maximality of $C$.)  Exactly the same is true of extensions $B_\mathrm{ext}$ of $B_\mathrm{int}=B\cap C$.  In short, the set of possible extensions to $B_\mathrm{int}$ and $R_\mathrm{int}$ are identical. Informally, $R_\mathrm{ext}$ is equally valid as an extension of $B_\mathrm{int}$ as $B_\mathrm{ext}$ is.  Thus $B=B_\mathrm{int}\cup R_\mathrm{ext}$ is a true sample from $\mu_{G,\lambda}^{(v,0)}$.
\end{proof} 

\begin{algorithm}
\caption {Growing a red-blue cluster by Breadth-First Search (BFS).}\label{alg:BFS}
\begin{algorithmic}
\STATE {$\textsf{Cluster BFS}(v)$}
\COMMENT{$R$ and $B$ are random independent sets, revealed as the algorithm progresses.}
\STATE {$L_0 \leftarrow \{v\}$}
\STATE {$U\leftarrow V(G)\setminus \{v\}$}
\COMMENT {the set of unprocessed vertices}
\STATE {$d\leftarrow0$}
\COMMENT {current depth of search}
\WHILE {$L_d\not=\emptyset$}
\STATE {$N\leftarrow U\cap\partial L_d$} 
\COMMENT {$\partial L_d=\{u\in V(G)\setminus L_d:\{u,w\}\in E(G)\text{ for some }w\in L_d\}$}
\IF {$d$ is even}
        \STATE {reveal the blue vertices within $N$, and set $L_{d+1}\leftarrow N\cap B$}
\ELSE    
        \STATE {reveal the red vertices within $N$, and set $L_{d+1}\leftarrow N\cap R$}
\ENDIF
\STATE {$U \leftarrow U\setminus N$}
\STATE {$d\leftarrow d+1$}
\ENDWHILE
\STATE {$d\leftarrow d-1$}
\COMMENT {The red-blue cluster containing $v$ is $C=L_0\cup L_1\cup\cdots\cup L_d$}
\end{algorithmic}
\end{algorithm}

Using this coupling we can rederive the result of Chen and Gu.

\begin{corollary}\label{cor:clawW}
Suppose that $\Delta\geq3$ and $\lambda>0$. If $G$ is a claw-free graph of maximum degree $\Delta$ and $v\in V(G)$, then $W_1(\mu_{G,\lambda}^{(v,0)},\mu_{G,\lambda}^{(v,1)})=O(1)$.
\end{corollary}

\begin{proof}
The claimed bound on $W_1$ is evidenced by the coupling just analysed.   As $R\oplus B=C$, we just have to demonstrate that $|C|$ is small in expectation.  Note that $G[C]$ is bipartite. It is straightforward to see that a connected bipartite claw-free graph can have no vertices of degree 3 or greater, so is either a path or a cycle.   As paths and cycles are one-dimensional structures, we do not expect them to propagate far, and this is indeed the case.  Imagine that we reveal the colours of vertices in a breadth-first fashion, starting with $v$ at time~0.  The breadth-first search is presented as Algorithm~\ref{alg:BFS}.  This is a little heavy for current purposes, but its generality will be useful later.  As the cluster is either a path or a cycle, we have $|L_j|\leq2$ for all levels~$j$. Since the hard core distribution is dominated by the product of independent Bernoullli distributions (Lemma~\ref{lem:sd}), on each iteration, with probability at least $p=(1/(\lambda+1))^{2\Delta}$, the breadth-first search will halt.  So $\Ex|C|=\Ex[|R\oplus B|]$ is bounded above by the expectation of a random variable having a geometric distribution with success probability $(\lambda+1)^{-2\Delta}$, which is $(\lambda+1)^{2\Delta}$.
\end{proof}

Then, by Corollary~\ref{cor:mixing} we obtain:

\begin{corollary}\label{cor:clawmixing}
For any $\Delta\geq3$ and $\lambda>0$, Glauber dynamics at fugacity $\lambda$ has optimal (i.e., $O(n\log n)$) mixing time on the class of claw-free graphs of maximum degree~$\Delta$.
\end{corollary}

Let's make some modest steps towards excluding larger subdivided claws.\footnote{The remainder of the section is not essential to the technical development of the main result, Theorem~\ref{thm:main}, relating to general subdivided claws.  However, it contains some contextual material on graphs with excluded subdivided claws, and also motivates the approach taken in Section~\ref{sec:general}.}  Fork-free graphs were characterised by Alekseev~\cite{Alekseev}.  See also Bencs~\cite{BencsPhD}. A connected bipartite fork-free graph is either a path, a cycle, or a complete bipartite graph minus a (possibly empty) matching.  Since the last of these can have size at most $2(\Delta+1)$, the argument just given in the Proof of Lemma~\ref{lem:couple} for claw-free graphs carries over to fork-free graphs, as the expected size of a cluster is still constant.  However, as we already saw in Corollary~\ref{cor:finitevariation}, this generalisation is illusory.

\begin{figure}
\begin{tikzpicture}[scale=0.5]
    \foreach \i in {0,...,2} {
        \foreach \j in {0,...,1} {
            \draw (4*\i,4*\j) node [vertex] (\i A\j) {};
            \draw (4*\i+1,4*\j+1) node [vertex] (\i B\j) {};
            \draw (4*\i+2,4*\j) node [vertex] (\i C\j) {};
            \draw (4*\i+1,4*\j-1) node [vertex] (\i D\j) {};

            \draw (\i A\j) -- (\i B\j) -- (\i C\j) -- (\i D\j) -- (\i A\j);
            \draw (\i A\j) -- (\i C\j);
            \draw (\i B\j) -- (\i D\j);
  
            \draw (4*\i+3,4*\j+1) node [vertex] (\i E\j) {};
            \draw (4*\i+4,4*\j) node [vertex] (\i F\j) {};
            \draw (4*\i+3,4*\j-1) node [vertex] (\i G\j) {};
            
            \draw (\i C\j) -- (\i E\j) -- (\i F\j) -- (\i G\j) -- (\i C\j);

            \draw (4*\i+0,4*\j+2) node [vertex] (\i H\j) {};
            \draw (4*\i+1,4*\j+3) node [vertex] (\i I\j) {};
            \draw (4*\i+2,4*\j+2) node [vertex] (\i J\j) {};
            
            \draw (\i B\j) -- (\i H\j) -- (\i I\j) -- (\i J\j) -- (\i B\j);
           }
    }
    
    \draw (1,9) node {$\vdots$};
    \draw (5,9) node {$\vdots$};
    \draw (9,9) node {$\vdots$};
    
    \draw (14,0) node {$\dots$};
    \draw (14,4) node {$\dots$};
    
    \draw [color=red, thick] (1A1) -- (1C1);
    \draw [color=red, thick] (1C1) -- (1E1);
    \draw [color=red, thick] (1C1) -- (1G1);
    \draw [color=red, thick] (0G1) -- (0F1);

\end{tikzpicture}
\caption{An (infinite family of) E-free but not fork-free graph(s).}
\label{fig:E-free}
\end{figure}

In light of Corollary~\ref{cor:finitevariation}, the first interesting case beyond claw free is E-free (i.e., $S_{1,2,2}$-free).   Figure~\ref{fig:E-free} illustrates an infinite family of E-free graphs which are not claw free.   Indeed, by iterating the basic building block either horizontally or vertically, or both, we can appreciate that the family includes graphs that contain induced copies of $S_{1,1,t}$ for arbitrarily large~$t$.  (An instance of a fork is picked out in red in Figure~\ref{fig:E-free}.) So, by excluding E, we have escaped from the constraints of Corollary~\ref{cor:finitevariation} into new, apparently interesting families of graphs. 

It is worth recording the fact that there is no analogue of Corollary~\ref{cor:finitevariation} for the sequence of $S_{1,2,t}$-free graph classes.  For example, suppose we construct a graph~$G$ as a disjoint union of an E together with a path of arbitrary length, with endpoints $a$ and $b$, and then add edges from $a$ to all vertices in the~E.  We claim that $G$ is skew-star-free (i.e., $S_{1,2,3}$-free).  Suppose to the contrary that $G$ contains an induced skew-star. This induced graph would have to include vertex $a$ together with at least one edge $e$ from the E.  But the endpoints of~$e$ together with~$a$ induce a triangle, which is a contradiction.     So there are infinitely many skew-star free graphs of degree at most 7 that are not E-free.

At this point we skip the class of E-free graphs and proceed directly to the larger class of skew-star-free graphs as these have been analysed in detail by Lozin~\cite{LozinSkew} in his `Decomposition Theorem'.  He showed that a connected bipartite skew-star-free graph is of one of four kinds.  Three of these are of bounded size, while the fourth is composed of blow-ups of paths and cycles.  (A blow-up of a graph~$H$ is obtained by replacing each vertex of~$H$ by an independent set, and each edge by a complete bipartite graph connecting the corresponding independent sets.) As the unbounded components are still `linear', we continue to have spectral independence and optimal mixing of Glauber dynamics.  The only change to the proof of Corollary~\ref{cor:clawW} lies in the analysis of the level sets $L_i$ generated by the breadth-first search:  we now have $|L_i|\leq2(\Delta-1)$ rather than $|L_j|\leq2$.

\begin{corollary}
Suppose that $\Delta\geq3$ and $\lambda>0$. If $G$ is a skew-star-free graph of maximum degree $\Delta$ and $v\in V(G)$, then $W_1(\mu_{G,\lambda}^{(v,0)},\mu_{G,\lambda}^{(v,1)})=O(1)$.
\end{corollary}

\begin{corollary}
For any $\Delta\geq3$ and $\lambda>0$, Glauber dynamics at fugacity $\lambda$ has optimal (i.e., $O(n\log n)$) mixing time on the class of skew-star-free (and hence E-free) graphs of maximum degree~$\Delta$.
\end{corollary}

As we saw earlier, there is no direct analogue of Lemma~\ref{lem:S111impliesS11t} with $S_{1,2,2}$ replacing $S_{1,1,1}$, and $S_{1,2,t}$ replacing $S_{1,1,t}$.  But if we add the condition of bipartiteness we do have:

\begin{lemma}\label{lem:2tot}
Suppose that $G$ is a connected bipartite graph of maximum degree $\Delta$, and that $t\geq3$.  If $|V(G)|>\vol(\Delta,t+2)$ and $G$ contains an induced $S_{1,2,2}$, then $G$ contains an induced $S_{1,2,t}$.
\end{lemma}

This may be proved in a similar way to Lemma~\ref{lem:S111impliesS11t}, but we omit the proof as it is not on the direct line to the main result.

\begin{corollary}
For any $\Delta\geq3$, $\lambda>0$ and $t\geq2$, Glauber dynamics at fugacity $\lambda$ has optimal (i.e., $O(n\log n)$) mixing time on the class of $S_{1,2,t}$-free graphs of maximum degree~$\Delta$.
\end{corollary}

\begin{proof}[Sketch proof]
The cluster containing $v$ is bipartite, and hence, by Lemma~\ref{lem:2tot} and Lozin's decomposition theorem, is either of bounded size or is the blow-up of a path or cycle.  The layers $L_i$ created by the breadth-first search of Algorithm~\ref{alg:BFS} are of bounded size, so the depth of the search is again bounded by a geometric random variable.
\end{proof} 

We have seen that in the case of $S_{1,1,t}$-free graphs and $S_{1,2,t}$-free graphs the cluster generated by Algorithm~\ref{alg:BFS} is either of bounded size, or is `linear' (in this case a blow-up of either a path or a cycle).  We may suspect that this remains true for graphs excluding larger subdivided stars.  While this is morally true, it is difficult to provide a suitable definition of `linear' that meshes with the coupling argument.  We therefore slightly change our approach.

\section{$S_{2,2,2}$-free graphs and beyond}\label{sec:general}

Instead of trying to analyse the structure of bounded-degree bipartite graphs excluding a large subdivided claw, we will analyse the BFS procedure of Algorithm~\ref{alg:BFS} directly. Recall the layers $L_i$ that are produced by the search.  The key observation is that these layers do not grow large. 

\begin{lemma}
If $G$ is $S_{t,t,t}$-free then $|L_i|\leq 2\vol(\Delta,2t)$, for all $i\in\mathbb{N}$ for which $L_i$ is defined. 
\end{lemma}

\begin{proof}
Suppose to the contrary that $|L_i|>2\vol(\Delta,2t)$.  Choose vertices $a,b,c\in L_i$ with the property that $d_G(a,b)$, $d_G(a,c)$ and $d_G(b,c)$ are all greater than $2t$.  This is always possible: select $a$ arbitrarily, and eliminate from consideration the at most $\vol(\Delta,2t)$ vertices in $L_i$ that are within distance $2t$ from $a$; choose $b$ from the remaining vertices and eliminate the vertices in $L_i$ that are within distance $2t$ from $b$; at least one choice for~$c$ remains.

Now orient all edges away from $v$, so that each edge goes from level $L_j$ to $L_{j+1}$ for some $j$.  Choose a minimum cardinality directed Steiner tree $T$ connecting $v$ (as root) to all of $a,b,c$. Consider the graph $T'$ induced by $V(T)$.  Suppose that $T'$ has some edge not in $E(T)$, say $e=\{u_j,u_{j+1}\}$ where $u_j\in L_j$ and $u_{j+1}\in L_{j+1}$.  (Recall that $T'$ is bipartite, so there are no edges within a layer;  also the levels were constructed by breadth-first search, so no edges of $T'$ skip a layer.) 

First assume that $T$ has a single vertex $w$ such that either $w\not=v$ and $\deg(w)=4$ or $w=v$ and $\deg(w)=3$. Note that all vertices of $T$ other than $v,w,a,b,c$ have degree~2.  By adding $e$ to $E(T)$ and removing the edges of $T$ lying between $w$ and $u_{j+1}$ we would obtain a smaller tree, contradicting minimality of~$T$.  (Notice that we must remove at least two edges.)  So such an edge~$e$ does not exist.  By choice of $a,b,c$ we know that $d_T(w,a),d_T(w,b), d_T(w,c)\geq t+1$, so $G$ contains an induced $S_{t,t,t}$.  (Note that we are measuring graph distances within the tree~$T$ here.)

Next assume that $T$ has two vertices $w$ and $w'$, with $w'$ in a higher numbered level than $w$, and such that either (i)~$w\not=v$ and and $\deg(w)=\deg(w')=3$ or (ii)~$w=v$ and $\deg(w)=2$ and $\deg(w')=3$. Note that all vertices of $T$ other than $v,w,w',a,b,c$ have degree 2.  Suppose, without loss of generality, that $a$ and~$b$ are descendants of~$w'$.  With two exceptions, we may add $e$ to $E(T)$ and remove all the edges between $u_{j+1}$ and $w$, or $w'$ as appropriate, to obtain a smaller tree, contradicting minimality of~$T$.  Those exceptions occur when $w'\in L_j$, i.e., $u_j$ is in the same level as $w'$.  In that case, one edge is added and one removed and the size of the tree remains unchanged.  There are two candidate vertices for $u_{j+1}$; these are the vertices in $L_{j+1}\cap V(T)$, say $x$ and $y$, that are adjacent to $w'$.   Let $x$ (respectively, $y$) be the vertex on the path in $T$ from $w'$ to $a$ (respectively $b$).  To summarise: The graph induced by $V(T)$ is tree $T$, with the possible addition of one or two of the edges $\{u_j,x\}$ and $\{u_j,y\}$. (Refer to Figure~\ref{fig:exceptions}.)

\begin{figure}
\begin{tikzpicture}[scale=0.7, semithick]
      \draw (-0.5,0) node [vertex] (v) {} node [above] {$v$};
      \draw (2,0) node [vertex] (w) {} node [above] {$w$};
      \draw (5,1) node [vertex] (w') {} node [above left] {$w'$};
      \draw (5,-2) node [vertex] (uj) {} node [below left] {$u_j$};
      \draw (6.2,1.7) node [vertex] (x) {} node [below right] {$x$};
      \draw (6.2,0.3) node [vertex] (y) {} node [below right] {$y$};
      \draw (9,2.5) node [vertex] (a) {} node [right] {$a$};
      \draw (9,0) node [vertex] (b) {} node [right] {$b$};
      \draw (9,-3) node [vertex] (c) {} node [right] {$c$};

      \draw (w') -- (x);
      \draw (w') -- (y);
      
      \draw [dash dot] (v) -- (w);
      \draw [dash dot] (w) -- (w');
      \draw [dash dot] (w) -- (uj);
      \draw [dash dot] (uj) -- (c);
      \draw [dash dot] (x) -- (a);
      \draw [dash dot] (y) -- (b);
      
      \draw [dotted, thick, color=red] (uj) -- (x);
      \draw [dotted, thick, color=red] (uj) -- (y);
      
      \draw (-0.5,-4.4) node {$L_0$} node [above=5pt] {$\uparrow$};
      \draw (5,-4.4) node {$L_j$} node [above=5pt] {$\uparrow$};
      \draw (6.2,-4.4) node {$L_{j+1}$} node [above=5pt] {$\uparrow$};
      \draw (9,-4.4) node {$L_i$} node [above=5pt] {$\uparrow$};
  
\end{tikzpicture}
\caption{The two exceptional edges (shown red, dotted).  The solid lines are edges and the dot-dashed lines are paths.}
\label{fig:exceptions}
\end{figure}

By choice of $a,b,c$ we know that $d_T(w',a),d_T(w',b)\geq t+1$.  In the case that neither of the exceptional edges is present, we have an induced  $S_{t,t,t}$ with its centre at~$w'$ and with the path from $w'$ to $c$ routing via $w$. In the case of one exceptional edge, say $\{u_j,x\}$, we have an induced  $S_{t,t,t}$ with its centre at $x$ and with the path from $x$ to $b$ routing via $w'$, and the path  from $x$ to $c$ routing via $u_j$.  In the case of two exceptional edges, we have an induced  $S_{t,t,t}$ with its centre at $u_j$ and with the path from $u_j$ to $a$ routing via $x$, and the path  from $u_j$ to $b$ routing via $y$.  

As the above cases are exhaustive, we obtain a contradiction to $G$ being $S_{t,t,t}$-free.
\end{proof}

\begin{corollary}
Suppose that $\Delta\geq3$, $\lambda>0$ and $t\geq2$. If $G$ is a $S_{t,t,t}$-free graph of maximum degree~$\Delta$ and $v\in V(G)$, then $W_1(\mu_{G,\lambda}^{(v,0)},\mu_{G,\lambda}^{(v,1)})=O(1)$.
\end{corollary}

\begin{proof}
The levels of the breadth-first search remain bounded.  There is a constant probability of the search terminating at each iteration.
\end{proof}

\begin{theorem}\label{thm:main}
For any $\Delta\geq3$, $\lambda>0$ and $t\geq2$, Glauber dynamics for the hard-core distribution at fugacity~$\lambda$ has optimal (i.e., $O(n\log n)$) mixing time on the class of $S_{t,t,t}$-free graphs of maximum degree~$\Delta$.
\end{theorem}

\begin{proof}
Follows immediately using Corollary~\ref{cor:mixing}.
\end{proof}

Naturally, the theorem also covers $S_{i,j,k}$-free graphs by taking $t=\max\{i,j,k\}$.

\goodbreak
\section{Exponential-time mixing of $H$-free graphs\\when $H$ is neither a path nor a subdivided claw}\label{sec:exponential}

To complete the dichotomy, it only remains to show:

\begin{theorem}\label{thm:torpid}
Let $H$ be a connected graph that is neither a subdivided claw nor a path.  There exists a fugacity~$\lambda$ and an infinite family $\calG$ of $H$-free graphs of maximum degree~3 such that Glauber dynamics for the hard-core distribution on these graphs has exponential mixing time on~$\calG$.
\end{theorem}

\begin{proof}
Let $\alpha>0$ be sufficiently small. Given $n$, let $G$ be a cubic bipartite graph on $n+n$ vertices that is a $1+\alpha$ expander for sets up to size $\frac23n$.  By this we mean the following.   Let $V_L\cup V_R= V(G)$ be the bipartition of the vertex set of~$G$. Then, every subset $S\subset V_L$ of size at most $\frac23n$ is adjacent to at least $(1+\alpha)|S|$ vertices in $V_R$, and the same is true with $V_L$ and $V_R$ interchanged. Such graphs exist for some fixed $\alpha>0$ by a lemma due to Bassalygo~\cite{Bassalygo};  see Alon~\cite[Lemma 4.1]{Alon}.   Let $\ell$ be a number and let $G^*$ denote the $(2\ell+1)$-stretch of $G$; that is, $G^*$ is obtained from $G$ by subdividing each edge by $2\ell$ new vertices.  Fix $\ell$ sufficiently large that $G^*$ contains no induced~$H$. Let $V_L$ and $V_R$ continue to denote the vertices in~$G^*$ that are inherited from~$G$.

Let $\calI_{<}\cup\calI_{=} \cup\calI_{>}=\calI_G$ be the partition of the independent sets $\calI_G$ of $G$ given by:
$$
\calI_{\circ}=\{I\in\calI:|I\cap V_L|\circ|I\cap V_R|\},
$$
where $\circ\in\{<,=,>\}$.  Note that, to pass from any independent set in~$\calI_<$ to any independent set in~$\calI_>$, Glauber dynamics must pass through an independent set in~$\calI_=$.  To show an exponential lower bound on mixing time it is enough to verify that $\wt(\calI_=)\leq e^{-cn}\min\{\wt(\calI_<),\wt(\calI_>)\}$ for some $c>0$.\footnote{Here, we have extended $\wt(\cdot)$ from configurations to sets of configurations in the obvious way.} This inequality implies that the `conductance' of Glauber dynamics on $G^*$ is exponential small, and hence that the mixing time is exponentially large~\cite[Claim 7.14]{ETH}.

Note that $G^*$ is bipartite, with a balanced bipartition, and having $|V(G^*)|=(6\ell+2)n$ vertices.  By selecting all vertices on one side of the bipartition we obtain an independent set of size $(3\ell+1)n$ in $\calI_<$, and by selecting all vertices in the opposite side an independent set of the same size in $\calI_>$.  Just considering these two independent sets  we see that
$$
\wt(\calI_<),\wt(\calI_>)\geq \lambda^{(3\ell+1)n}.
$$

We will show that the independent sets in $\calI_=$ all have cardinality significantly less than $(3\ell+1)n$.  So if $\lambda$ is sufficiently large, the total weight $\wt(\calI_=)$ of independent sets in $\calI_=$ will be exponentially smaller than $\wt(\calI_<)$.  So consider an independent set $I$ with $|I\cap V_L|=|I\cap V_R|=k$.  In general, each of the paths of length $2\ell+1$ in~$G^*$ joining a vertex $u$ in $V_L$ to a vertex $v$ in $V_R$ is able to support an independent set of cardinality~$\ell$ on its internal vertices (the vertices that are neither in $V_L$ nor in $V_R$).  However, if $u$ and $v$ are both in~$I$ then the internal vertices can only support an independent set of size $\ell-1$. 

By expansion, if $n/(2+\alpha)\leq k\leq\frac23n$, there must be at least $(2+\alpha)k-n$ of these deficient paths.  Thus, when $0\leq k\leq\frac23n$, the cardinality of an independent set $I$ in $G^*$ with $|I\cap V_L|=|I\cap V_R|=k$ is at most 
$$
|I|=2k+3\ell n -\max\{ (2+\alpha)k-n,0\}.
$$
The maximum of this expression is achieved at $k=n/(2+\alpha)$ where it takes the value $(3\ell +2/(2+\alpha))n$.  When $k>\frac23n$ there is a simpler argument.  There are $3k$ edges of $G$ leaving the set $I\cap V_L$ and at most $3(n-k)$ of these can be accommodated in $V_R\setminus I$.  So there are at least $6k-3n$ deficient paths and the cardinality of an independent set In~$G^*$ in this range is at most $2k+3\ell n-(6k-3n)=(3\ell+3)n-4k$.  This expression is maximised at $k=\frac23n$, at which point it evaluates to $(3\ell+\frac13)n$.  Over the whole range of $k$, the maximum cardinality of an independent set in~$\calI_=$ is bounded by $(3\ell +2/(2+\alpha))n$.

As $|V(G^*)|=(6\ell+2)n$, the total weight of independent sets in $\calI_=$ is at most
$$
\wt(\calI_=)\leq 2^{(6\ell+2)n} \lambda^{(3\ell+2/(2+\alpha))n},
$$ 
and hence the ratio of $\wt(\calI_=)$ to $\wt(\calI_<)$ is at most
$$
\frac{\wt(\calI_=)}{\wt(\calI_<)}\leq \big(2^{6\ell+2} \lambda^{2/(2+\alpha)-1}\big)^n.
$$ 
By setting $\lambda$ sufficiently large in terms of $\ell$ (which in turn is determined by $H$) we can ensure that the above ratio is at most~$2^{-n}$.  Thus, Glauber dynamics on $G^*$ has exponentially small conductance, and exponentially large mixing time.  
\end{proof}


\begin{proof}[Proof of Theorem~\ref{thm:composite}]
The positive direction is immediate from Theorem~\ref{thm:main}.  
Now choose $\ell$ sufficiently large such that the stretched graph $G^*$ in the proof of Theorem~\ref{thm:torpid} contains none of the graphs $H\in\mathcal{H}$.  The construction of Theorem~\ref{thm:torpid} yields a fugacity~$\lambda$ and an infinite family of graphs (both depending on $\ell$) on which Glauber dynamics requires exponential mixing time at fugacity~$\lambda$.
\end{proof}

\section*{Acknowledgements}
I thank Heng Guo and Viresh Patel for useful pointers.

\bibliographystyle{plain}
\bibliography{H-free}

\end{document}